\documentclass[11pt]{amsart}

\usepackage{amsmath,xspace,amssymb,mathrsfs}
\usepackage{color}

\input xy
\xyoption{all}
\xyoption{2cell}
\UseAllTwocells
\CompileMatrices

\newcommand{\Spec}{\operatorname{Spec}}
\renewcommand{\phi}{\varphi}

\newcommand{\Ker}{\operatorname{Ker}}

\newcommand{\Ima}{\operatorname{Im}}

\newcommand{\Max}{\operatorname{Max}}

\newcommand{\Min}{\operatorname{Min}}

\newcommand{\Supp}{\operatorname{Supp}}

\newcommand{\Sp}{\operatorname{Sp}}

\newtheorem{proposition}{Proposition}[section]
\newtheorem{lemma}[proposition]{Lemma}

\newtheorem{corollary}[proposition]{Corollary}
\newtheorem{theorem}[proposition]{Theorem}

\theoremstyle{definition}
\newtheorem{definition}[proposition]{Definition}
\newtheorem{example}[proposition]{Example}

\newtheorem{remark}[proposition]{Remark}

\usepackage{etoolbox}
\makeatletter
\patchcmd{\@settitle}{\uppercasenonmath\@title}{}{}{}
\patchcmd{\@setauthors}{\MakeUppercase}{}{}{}
\makeatother

\begin{document}

\title[Tame and wild primes]{Tame and wild primes in direct products of commutative rings}

\author[A. Tarizadeh]{Abolfazl Tarizadeh}
\address{Department of Mathematics, Faculty of Basic Sciences, University of Maragheh, Maragheh, East Azerbaijan Province, Iran.}
\email{ebulfez1978@gmail.com}

\date{}
\subjclass[2010]{14A05, 13A15, 13B30, 13B02}
\keywords{infinite direct product of rings; tame prime; wild prime; tame connected component; wild connected component \\
$^{\ast}$ Corresponding author.}

\begin{abstract} In this work, new progress has been made in understanding the structure of all prime ideals of infinite direct products of commutative rings. In particular, we observe that in an infinite direct product of nonzero rings there are two different types of prime ideals, that we call tame primes and wild primes. \\
Among the main results, we prove that the set of tame primes is an open subscheme of the prime spectrum, and this scheme is non-affine if and only if the index set is infinite. As an application, a prime ideal is a wild prime if and only if it contains the direct sum ideal. \\ 
Next, we show that an uncountable number of (wild) primes of an infinite direct product ring are induced by the (non-principal) ultrafilters of the index set (at least $2^{\mathfrak{c}}$ wild primes, $\mathfrak{c}$ is the cardinality of the continuum). This description has an important consequence which asserts that if a direct product ring has a wild prime, then the set of wild primes is infinite (uncountable). \\
The connected components of the prime spectrum of an infinite direct product ring are also investigated. We observe that, like for prime ideals, there are two types of connected components, tame ones and wild ones.
\end{abstract}

\maketitle

\section{Introduction}

If $p$ is a prime number and $n\geqslant1$, then $\mathbb{Z}/p^{n}\mathbb{Z}$ is a local zero-dimensional ring (i.e., its prime spectrum is a singleton), but quite surprisingly the Krull dimension of the direct product ring $\prod\limits_{n\geqslant1}\mathbb{Z}/p^{n}\mathbb{Z}$ is infinite, and this ring has a huge number of prime ideals (the cardinality of its set of maximal ideals is uncountable and equals $2^{\mathfrak{c}}$ where
$\mathfrak{c}$ is the cardinality of the continuum). However, the structure of all prime ideals of this ring as well as $\prod\limits_{n\geqslant1}\mathbb{Z}$ are not precisely understood yet. These examples show that fully understanding the structure of prime ideals of an infinite direct product of rings (even for very specific rings such as $\mathbb{Z}$ or $\mathbb{Z}/p^{n}\mathbb{Z}$) is a complicated problem. It is still unclear how the form of every prime ideal in an infinite direct product ring looks like. In the literature (see e.g. (
\cite{Gilmer-Heinzer}, \cite{Gilmer-Heinzer 2}, \cite{Gilmer-Heinzer 3}, \cite{Levy-Loustaunau-Shapiro} and \cite{Maroscia}), this problem is broadly considered and several interesting results have been obtained.

In this article we observe that in an infinite direct product ring there are two types of prime ideals, tame primes and wild primes. In Theorem \ref{Theorem 1 IDP}, we prove that the set of tame primes of a direct product ring $R=\prod\limits_{i\in S}R_{i}$ is an open subscheme of $\Spec(R)$, and this scheme is non-affine if and only if the index set $S$ is infinite. As an application, in Corollary \ref{Corollary I}, it is shown that a prime ideal of $R$ is a wild prime if and only if it contains the direct sum ideal $\bigoplus\limits_{i\in S}R_{i}$. In contrast to the wild primes, the structure of tame primes can be completely understood. Next in Theorem \ref{Theorem T-D}, we prove that very many wild primes (but not all) of an infinite direct product ring $R=\prod\limits_{i\in S}R_{i}$ can be described in terms of the non-principal ultrafilters of the index set $S$.
This description has a nice consequence (see Corollary \ref{Corollary 2 card}).
Then we show with an example that there are wild primes that do not fit into this description. In fact, the cardinality of the set of wild primes of an infinite direct product ring is strictly greater than the cardinality of the continuum. Note that despite the large number of wild primes, one cannot construct an explicit example of a wild prime in an infinite direct product ring, because the existence of them uses the axiom of choice and hence making it non-constructive. In Theorems \ref{Theorem 9 nine-dokuz} and \ref{Theorem 10 ten}, the cardinalities of the sets of minimal wild primes and maximal wild primes are determined. In an infinite direct product ring $R=\prod\limits_{k\in S}R_{k}$ the direct sum ideal $I=\bigoplus\limits_{k\in S}R_{k}$ plays an important role. Indeed, the problem of understanding the structure of prime ideals and several other properties of $R$ are reduced to understand the same property for the quotient ring $R/I$. Then we show that if each $R_{k}$ is a local ring then $R/I$ is canonically isomorphic to a certain localization of $R$ (see Theorem \ref{Theorem 8 eight}).

In \S3, the connected components of the prime spectrum of an infinite direct product ring are investigated. We observe that, like for prime ideals, there are two types of connected components, tame ones and wild ones. We show that if the index set $S$ is infinite, then the set of (wild) connected components of $\Spec(R)$ is uncountable (see Theorem \ref{Thm ccc} and Corollary \ref{Corollary 4 dort-four}). Next in Theorem \ref{Theorem tmri ng}, the structure of tame connected components is characterized.

We should add that although the inverse limit construction is indeed a generalization of the direct product construction, one cannot expect that the results which hold for direct products of rings also hold to inverse limits of rings as well (and vice versa). For example, let $p$ be a prime number. Then the ring of $p$-adic integers is the inverse limit of the rings $\mathbb{Z}/p^n\mathbb{Z}$ with $n\geqslant1$. It is well known that this ring is a local PID of Krull dimension one, and so it has precisely two prime ideals, the zero ideal and the ideal generated by $p$. But quite in contrast, the direct product of the rings $\mathbb{Z}/p^n\mathbb{Z}$ (which contains the ring of $p$-adic integers as a subring) is of infinite Krull dimension, it is not a domain and is a non-Noetherian ring and the number of its prime ideals is uncountable.

\section{Tame and wild primes}

In this article, all rings are assumed to be commutative. If $r$ is an element of a ring $R$, then $D(r)=\{\mathfrak{p}\in\Spec(R): r\notin\mathfrak{p}\}$ and $V(r)=\Spec(R)\setminus D(r)$. Every ring map $\phi:R\rightarrow R'$ induces a map $\phi^{\ast}:\Spec(R')\rightarrow\Spec(R)$
between the corresponding prime spectra which is given by $\mathfrak{p}\mapsto\phi^{-1}(\mathfrak{p})$. The following definition is the key notion of this article.

\begin{definition} Let $(R_{i})$ be a family of nonzero rings indexed by a set $S$ and let $R=\prod\limits_{i\in S}R_{i}$ be their direct product ring. If $\mathfrak{p}$ is a prime ideal of $R_k$ for some $k\in S$, then we call $\pi^{-1}_{k}(\mathfrak{p})$ a \emph{tame prime of $R$} where $\pi_{k}:R\rightarrow R_{k}$ is the projection map.
By a \emph{wild prime of $R$} we mean a prime ideal of $R$ which is not tame, i.e., it is not of the form $\pi^{-1}_{k}(\mathfrak{p})$.
\end{definition}

It is clear that $\pi^{-1}_{k}(\mathfrak{p})=\prod\limits_{i\in S}I_{i}$ where $I_{k}=\mathfrak{p}$ and $I_{i}=R_{i}$ for all $i\neq k$.  Throughout this article, for each $k\in S$, we call the element $e_{k}:=(\delta_{i,k})_{i\in S}$ the $k$-th unit idempotent of $R$ where $\delta_{i,k}$ is the Kronecker delta. The following result gives us information about the geometric aspects of tame primes.

\begin{theorem}\label{Theorem 1 IDP} If $R=\prod\limits_{i\in S}R_{i}$ is the direct product of nonzero rings $R_{i}$, then the set of tame primes of $R$ is an open subscheme of $\Spec(R)$. This scheme is affine if and only if the index set $S$ is finite.
\end{theorem}

\begin{proof} Let $T$ be the set of tame primes of $R$. We first show that $T=\bigcup\limits_{k\in S}D(e_{k})$ where $e_{k}$ is the $k$-th unit idempotent of $R$ for all $k$. The inclusion $T\subseteq\bigcup\limits_{k\in S}D(e_{k})$ is obvious. To see the reverse inclusion, let $P$ be a prime ideal of $R$ with $P\in\bigcup\limits_{k\in S}D(e_{k})$. Then $e_{k}\notin P$ for some $k$. It follows that $e_{i}\in P$ for all $i\neq k$, because $e_{i}e_{k}=0$. We claim that $\pi_{k}(P)$ is a prime ideal of $R_{k}$. Clearly it is an ideal of $R_{k}$, since $\pi_{k}$ is surjective. It is also a proper ideal of $R_{k}$, because if $1\in\pi_{k}(P)$ then there exists some $a=(a_{i})\in P$ such that $a_{k}=1$, then $e_{k}=e_{k}a\in P$ which is a contradiction. Now suppose there are $a,b\in R_{k}$ such that $ab\in\pi_{k}(P)$. Then there exists some $x=(x_{i})\in P$ such $x_{k}=ab$. Then consider the elements $y=(y_{i})$ and $z=(z_{i})$ in $R$ where $y_{i}=1$ and $z_{i}=x_{i}$ for all $i\neq k$ and $y_{k}=a$ and $z_{k}=b$. Then $yz=x\in P$. Thus $y\in P$ or $z\in P$. It follows that $a\in\pi_{k}(P)$ or $b\in\pi_{k}(P)$. This establishes the claim. Now we show that $P=\pi^{-1}_{k}(\mathfrak{q})$ where $\mathfrak{q}:=\pi_{k}(P)$. The inclusion $P\subseteq\pi^{-1}_{k}(\mathfrak{q})$ is obvious. If $r=(r_{i})\in\pi^{-1}_{k}(\mathfrak{q})$ then $r_{k}\in\mathfrak{q}$. So there exists some $b=(b_{i})\in P$ such that $b_{k}=r_{k}$. Then $e_{k}r=e_{k}b\in P$. Hence, $r\in P$. Thus $P$ is a tame prime of $R$, i.e., $P\in T$. Therefore
$T=\bigcup\limits_{k\in S}D(e_{k})$ is an open subset of $\Spec(R)$. If the index set $S$ is finite, then it can be seen that $T=\Spec(R)$ is an affine scheme. Conversely, there exists a natural number $n\geqslant1$ such that $T=\bigcup\limits_{k=1}^{n}D(e_{i_{k}})=D(e)$ where $e:=\sum\limits_{k=1}^{n}e_{i_{k}}$, because the underlying space of every affine scheme is quasi-compact.  If $S$ is infinite, then we may choose some $d\in S\setminus\{i_{1},\ldots,i_{n}\}$. Thus $e_{d}e=0$ and so $D(e_{d})=D(e_{d})\cap D(e)=D(e_{d}e)=D(0)$. Since $e_{d}$ is an idempotent, we get that $e_{d}=0$. But this is a contradiction, because $R_{d}$ is a nonzero ring.
\end{proof}

The above theorem has several important consequences:

\begin{corollary}\label{Corollary I} The set of wild primes of the ring $R=\prod\limits_{i\in S}R_{i}$ is the closed subset $V(I)$ of $\Spec(R)$ where $I=\bigoplus\limits_{i\in S}R_{i}$.
\end{corollary}

\begin{proof} Let $X$ be the set of wild primes of $R$. Then by Theorem \ref{Theorem 1 IDP} and its proof, $X=\bigcap\limits_{i\in S}V(e_{i})$. But the direct sum ideal $I=\bigoplus\limits_{i\in S}R_{i}$ is generated by the $e_{i}$, i.e. $I=(e_{i}: i\in S)$. Hence, $V(I)=\bigcap\limits_{i\in S}V(e_{i})=X$.
\end{proof}

\begin{remark} It is well known that if we have finitely many pairwise disjoint affine opens $U_k=\Spec(A_k)$ (with $k=1,\ldots,n$) in a scheme $X$, then their union $U:=\bigcup\limits_{k=1}^{n}U_k$ is an affine open. Indeed, the open subscheme $U$ is isomorphic to the affine spectrum of the direct product ring $\prod\limits_{k=1}^{n}A_k$.
However, this fact does not necessarily hold for an infinite number of affine opens. For example, consider the open subscheme in Theorem \ref{Theorem 1 IDP}.
\end{remark}

\begin{example} Let $X$ be a scheme with the structure sheaf $\mathscr{O}_{X}$ and $f\in\mathscr{O}_{X}(X)$ a global section. If $X$ is an affine scheme, then the principal open subset $X_f=\{x \in X : f_x\notin\mathfrak{m}_{x}\}$ is an affine open of $X$. But in general, the open subscheme $X_f$ is no longer an affine open of the scheme $X$. For instance, $X_1=X$. It seems a bit difficult to find other examples of this type that are nontrivial (i.e. $f\neq1$). But we will give an interesting basic example. To this end, let $R=\prod\limits_{k\geqslant1}\mathbb{Z}$ and $X=\Spec(R)$. Let $U=\bigcup\limits_{k\geqslant1} D(e_k)$ and $W=\bigcup\limits_{k\geqslant1} D(e_{2k})$. We know that the open subschemes $W\subset U$ are non-affine (since their underlying spaces are not quasi-compact). Then consider $f=(0,1,0,1,...) \in R$ (i.e. every odd component of $f$ is zero and every even component is 1) and let $g=f_{|_{U}}$ which is an element of $\mathscr{O}_X (U)$. Then we have $W=U_g$. Indeed, if $\mathfrak{p}\in W$ then $fe_{2k}=e_{2k}\notin\mathfrak{p}$ for some $k$ and so $f\notin\mathfrak{p}$. This shows that $g_{\mathfrak{p}}=f_{\mathfrak{p}}=f/1\notin
\mathfrak{p}R_{\mathfrak{p}}$ and so $\mathfrak{p}\in U_g$. To see the reverse inclusion, take $\mathfrak{p}\in U_g$. Thus $f\notin\mathfrak{p}$. If $\mathfrak{p}\notin W$ then $e_{2k}\in\mathfrak{p}$ for all $k\geqslant1$. But $U_g\subseteq U$ and so $e_d\notin\mathfrak{p}$ for some odd number $d\geqslant1$. But $0=fe_{d}\notin\mathfrak{p}$ which is a contradiction. Hence, $W=U_g$.
\end{example}

\begin{remark} If $P$ is a tame prime of $R=\prod\limits_{i\in S}R_{i}$ then $P=\prod\limits_{i\in S}\mathfrak{p}_{i}$ where $\mathfrak{p}_{k}:=\pi_{k}(P)$ is a prime ideal of $R_{k}$ for some $k\in S$ and $\mathfrak{p}_{i}=R_{i}$ for all $i\neq k$. In this case, we also have the following canonical isomorphisms of rings: $R/P\simeq R_{k}/\mathfrak{p}_{k}$ and $R_{P}\simeq(R_{k})_{\mathfrak{p}_{k}}$. Also note that if $P'$ is a wild prime of $R$ then $P+P'=R$. This shows that for any totally ordered set of prime ideals of $R$, if one of them is a tame (resp. wild) prime then all elements of this set are tame (resp. wild) primes.
\end{remark}

\begin{corollary}\label{coro 5 wno} If the index set $S$ is infinite, then the set of wild primes of $R=\prod\limits_{i\in S}R_{i}$ is not an open subset of $\Spec(R)$.
\end{corollary}

\begin{proof} If the set of wild primes of $R$ is an open subset of $\Spec(R)$, then by Corollary \ref{Corollary I} and \cite[Theorem 1.1]{Tarizadeh-Sharma}, there exists an idempotent $e\in R$ such that $V(I)=V(e)$ where $I=\bigoplus\limits_{i\in S}R_{i}$. It follows that $\sqrt{I}=\sqrt{Re}$. For each $k\in S$, we have $e_{k}\in I$ and so $e_{k}\in\sqrt{Re}$.
Thus there exists some $d\geqslant1$ such that $e_{k}=(e_{k})^{d}\in Re$. So $e_{k}=re$ for some $r\in R$. It follows that $e_{k}(1-e)=re(1-e)=0$ and so $e_{k}=e_{k}e$. Thus $e=1$ is the unit element of $R$. It follows that $I=R$ which is a contradiction, because the index set $S$ is infinite (and each $R_{k}\neq0$) and so $I\neq R$.
\end{proof}

\begin{remark} We know that every finitely generated flat module admits the rank map (because every finitely generated flat module over a local ring is a free module). It is well known that if the rank map of a finitely generated flat module is locally constant (i.e. it is a constant map in an open neighbourhood of every prime ideal), then it is a projective module. Note that every locally constant map over a quasi-compact space takes only finitely many values (but a map with finitely many values is not necessarily locally constant). We also know that many of the finitely generated flat modules are projective (see e.g. \cite{Tanri}). But here we give an example of a finitely generated flat module which is not projective and at the same time its rank map takes only finitely many values.
The direct sum ideal $I=\bigoplus\limits_{i\in S}R_{i}$ of $R=\prod\limits_{i\in S}R_{i}$ is generated by the idempotents $e_{k}$, and so $R/I$ is a flat $R$-module. It is well known that the annihilator of every finitely generated projective module is generated by an idempotent element. This shows that if the index set $S$ is infinite then $R/I$ is a not a projective $R$-module. The rank map of the finitely generated flat $R$-module $R/I$ takes only the values 0 and 1 over $\Spec(R)$. Indeed, if $\mathfrak{p}\notin V(I)$ then $(R/I)_{\mathfrak{p}}=0$. If $\mathfrak{p}\in V(I)$ then $(R/I)_{\mathfrak{p}}\simeq R_{\mathfrak{p}}/IR_{\mathfrak{p}}\simeq R_{\mathfrak{p}}$, because for each $k$, $1-e_{k}\notin\mathfrak{p}$ and so $IR_{\mathfrak{p}}=0$.
\end{remark}

If for each $k\in S$, $I_{k}$ is an ideal of a ring $R_{k}$, then $\bigoplus\limits_{k\in S}I_{k}\subseteq\prod\limits_{k\in S}I_{k}$ are ideals of $R=\prod\limits_{k\in S}R_{k}$, and the quotient ring $R/\prod\limits_{k\in S}I_{k}$ is canonically isomorphic to $\prod\limits_{k\in S}R_{k}/I_{k}$. If $\phi:A=\prod\limits_{i\in S}A_{i}\rightarrow B=\prod\limits_{i\in S}B_{i}$ is the ring map induced by a family of ring maps $(\phi_{i}:A_{i}\rightarrow B_{i})$, then $\Ker(\phi)=\prod\limits_{i\in S}\Ker(\phi_{i})$ and $\Ima(\phi)=\prod\limits_{i\in S}\Ima(\phi_{i})$. In particular, $\phi$ is injective if and only if each $\phi_{i}$ is injective. Similarly, $\phi$ is surjective if and only if each $\phi_{i}$ is surjective.

\begin{corollary}\label{Corollary 3} Let $\phi:A=\prod\limits_{i\in S}A_{i}\rightarrow B=\prod\limits_{i\in S}B_{i}$ be the ring map induced by a family of ring maps $(\phi_{i}:A_{i}\rightarrow B_{i})$. Then a prime ideal $\mathfrak{p}$ of $B$ is a wild prime if and only if $\phi^{-1}(\mathfrak{p})$ is a wild prime of $A$.
\end{corollary}

\begin{proof} If $\mathfrak{p}$ is a wild prime of $B$, then by Corollary \ref{Corollary I}, $\bigoplus\limits_{i\in S}B_{i}\subseteq\mathfrak{p}$. It is obvious that the image of each $k$-th unit idempotent $e_{k}$ of $A$ under $\phi$ is the corresponding $k$-th unit idempotent of $B$, i.e., $\phi(e_{k})=e'_{k}$. This yields that $\bigoplus\limits_{i\in S}A_{i}\subseteq\phi^{-1}(\bigoplus\limits_{i\in S}B_{i})\subseteq\phi^{-1}(\mathfrak{p})$. Then, again by  Corollary \ref{Corollary I}, $\phi^{-1}(\mathfrak{p})$ is a wild prime of $A$. Conversely, suppose $\mathfrak{p}$ is a tame prime of $B$. So $\mathfrak{p}=\pi^{-1}_{k}(\mathfrak{p}_{k})$ where $\mathfrak{p}_{k}$ is a prime ideal of $B_{k}$ for some $k$. Then using the following commutative diagram:
$$\xymatrix{
A\ar[r]^{\phi}\ar[d]^{\pi'_{k}} &B\ar[d]^{\pi_{k}}\\A_{k}\ar[r]^{\phi_{k}}&B_{k}}$$ we get that $\phi^{-1}(\mathfrak{p})=(\pi'_{k})^{-1}\big(\phi_{k}^{-1}(\mathfrak{p}_{k})\big)$ is a tame prime of $A$ which is a contradiction.
\end{proof}

\begin{definition}\label{Def 1} If a function $f:\Spec(A)\rightarrow\Spec(B)$ has the property that a prime ideal of $A=\prod\limits_{i\in S}A_{i}$  is a wild prime  if and only if its image under $f$ is a wild prime of $B=\prod\limits_{k\in S'}B_{k}$ then we say that $f$ \emph{preserves wild primes}.
\end{definition}

It is easy to check that a map $f:\Spec(A)\rightarrow\Spec(B)$ preserves wild primes if and only if $f$ preserves tame primes (in a similar sense).

\begin{example} The map $\Spec(B)\rightarrow\Spec(A)$ induced by $\phi$ in Corollary \ref{Corollary 3} preserves wild primes. We also give examples of maps that do not preserve wild primes. Assume the index set $S$ is infinite, so we may choose a wild prime $\mathfrak{q}$ in the ring $R=\prod\limits_{i\in S}R_{i}$. Consider the canonical ring map $\phi:R\rightarrow R'=\prod\limits_{\mathfrak{p}\in\Spec(R)}R/\mathfrak{p}$ given by $r\mapsto(r+\mathfrak{p})_{\mathfrak{p}\in\Spec(R)}$. Now the map $\Spec(R')\rightarrow\Spec(R)$ induced by $\phi$ does not preserve wild primes, because $P:=\pi^{-1}(0)$ is a tame prime of $R'$ where $\pi:R'\rightarrow R/\mathfrak{q}$ is the projection map, but $\phi^{-1}(P)=\mathfrak{q}$ since $\pi\phi:R\rightarrow R/\mathfrak{q}$ is the canonical map. As a second example, consider the (injective) canonical ring map $\psi:R\rightarrow R'=\prod\limits_{\mathfrak{p}\in\Spec(R)}R_{\mathfrak{p}}$ given by $r\mapsto(r/1)_{\mathfrak{p}\in\Spec(R)}$. Then the map $\Spec(R')\rightarrow\Spec(R)$ induced by $\psi$ does not preserve wild primes, because $P:=\pi^{-1}(\mathfrak{q}R_{\mathfrak{q}})$ is a tame prime of $R'$ where $\pi:R'\rightarrow R_{\mathfrak{q}}$ is the projection map, but $\psi^{-1}(P)=\mathfrak{q}$ since $\pi\psi:R\rightarrow R_{\mathfrak{q}}$ is the canonical map.
\end{example}

For a set $S$ by $\mathcal{P}(S)\simeq\prod\limits_{k\in S}\mathbb{Z}_{2}$ we mean the power set ring of $S$ which is a Boolean ring where $\mathbb{Z}_{2}=\{0,1\}$ is the field of integers modulo two. We know that in a Boolean ring, every prime ideal is a maximal ideal. If $S$ is finite then the set of prime ideals of $\mathcal{P}(S)$ has the cardinality $|S|$. But if $S$ is infinite, then the set of prime ideals of $\mathcal{P}(S)$ has the cardinality $2^{2^{|S|}}$.
For more information on the power set ring we refer the interested reader to \cite{Tarizadeh-Taheri}. The following result shows that many (but not all) wild primes of an infinite direct product ring $R=\prod\limits_{i\in S}R_{i}$ can be described in terms of the non-principal ultrafilters of the index set $S$.

\begin{theorem}\label{Theorem T-D} If $R=\prod\limits_{i\in S}R_{i}$ is the direct product of nonzero rings $R_{i}$, then we have a continuous imbedding $\Spec(\mathcal{P}(S))\rightarrow\Spec(R)$ which preserves wild primes.
\end{theorem}

\begin{proof} Each $R_{i}$ is a nonzero ring, so it has at least a prime ideal $\mathfrak{p}_{i}$. For any  $a=(a_{i})\in R=\prod\limits_{i\in S}R_{i}$ setting $a^{\ast}:=\{i\in S: a_{i}\in\mathfrak{p}_{i}\}$. If $M$ is a prime (maximal) ideal of the power set ring $\mathcal{P}(S)$, or equivalently, $\mathcal{P}(S)\setminus M$ is an ultrafilter of $S$, then we first show that $M^{\ast}:=\{a\in R: a^{\ast}\notin M\}$ is a prime ideal of $R$. For the zero element 0 we have $0^{\ast}=S\notin M$ thus $0\in M^{\ast}$, and so $M^{\ast}$ is nonempty. If $a,b\in M^{\ast}$ then $a+b\in M^{\ast}$, because $a^{\ast}\cap b^{\ast}\subseteq (a+b)^{\ast}$ and so $(a+b)^{\ast}\notin M$. If $r\in R$ and $a\in M^{\ast}$ then $ra\in M^{\ast}$, since $a^{\ast}\subseteq r^{\ast}\cup a^{\ast}=(ra)^{\ast}$ and so $(ra)^{\ast}\notin M$. We also have $1^{\ast}=\emptyset\in M$ and so $1\notin M^{\ast}$.
Thus $M^{\ast}$ is a proper ideal of $R$. Now, suppose $ab\in M^{\ast}$ for some $a,b\in R$. If none of $a$ and $b$ is a member of $M$, then $a^{\ast}\in M$ and $b^{\ast}\in M$, and so $(ab)^{\ast}=a^{\ast}\cup b^{\ast}=a^{\ast}+b^{\ast}+a^{\ast}\cap b^{\ast}\in M$, a contradiction. Hence, $M^{\ast}$ is a prime ideal of $R$.
Next we show that the map
$\Spec(\mathcal{P}(S))\rightarrow\Spec(R)$ given by $M\mapsto M^{\ast}$ is injective. Suppose $M^{\ast}=N^{\ast}$. If $A\in M$ then consider the element $a=(a_{i})$ where $a_{i}=1$ for all $i\in A$ and otherwise $a_{i}=0$. Then $a^{\ast}=S\setminus A\notin M$ thus $a\in M^{\ast}=N^{\ast}$ and so $a^{\ast}\notin N$. This yields that $A\in N$. It follows that $M\subseteq N$ and so $M=N$. This map is also continuous, because the inverse image of $D(a)$ under this map equals $V(a^{\ast})=D(1-a^{\ast})$.
Finally, we show that this map preserves wild primes. If $M$ is a wild prime (i.e., non-principal maximal ideal) of $\mathcal{P}(S)$, then $\{k\}\in M$ for all $k\in S$. For the $k$-th unit idempotent $e_{k}$ we have $(e_{k})^{\ast}=S\setminus\{k\}\notin M$ thus $e_{k}\in M^{\ast}$ and so $\bigoplus\limits_{k\in S}R_{k}\subseteq M^{\ast}$. Then by Corollary \ref{Corollary I}, $M^{\ast}$ is a wild prime of $R$. Conversely, let $M^{\ast}$ be a wild prime of $R$ where $M$ is a maximal ideal of $\mathcal{P}(S)$. If $M$ is a tame prime of $\mathcal{P}(S)$ then $M=\mathcal{P}(S\setminus\{k\})$ for some $k\in S$. It follows that $M^{\ast}=\pi^{-1}_{k}(\mathfrak{p}_{k})$ which is a contradiction.
\end{proof}

\begin{example}\label{Remarh i} In Theorem \ref{Theorem T-D}, we observed that many (uncountable number of) wild primes of an infinite direct product ring $R=\prod\limits_{i\in S}R_{i}$ can be described in terms of the non-principal ultrafilters of $S$ (i.e., wild primes of the power set ring of $S$). However, note that every wild prime of $R$ does not necessarily fit into this description. For example, consider $R=\prod\limits_{n\geqslant1}\mathbb{Z}/p^{n}\mathbb{Z}$ with $p$ a prime number. Then each $a_{n}=p+p^{n}\mathbb{Z}\in\mathbb{Z}/p^{n}\mathbb{Z}$ is nilpotent. Thus for every maximal ideal $M$ of $\mathcal{P}(S)$ with $S=\{1,2,3,\ldots\}$ then the element $(a_{1},a_{2},a_{3},\ldots)$ is a member of $M^{\ast}$. But this element is not nilpotent, and hence not contained in some prime ideal $P$ of $R$. So $P$ is a wild prime and it is not of the form $M^{\ast}$.
\end{example}

\begin{remark}\label{Remark ii} It is well known that every infinite cardinal $\kappa$ is idempotent: $\kappa\cdot\kappa=\kappa$ (in fact, this statement is equivalent to the axiom of choice). In particular, if $\alpha$ is an infinite cardinal and $\beta$ is a nonzero cardinal, then $\alpha+\beta=\alpha\cdot\beta=\max\{\alpha,\beta\}$.
Note that $B^{A}=$ ``the set of all functions from a set $A$ to a set $B$'' is contained in $\mathcal{P}(A\times B)$. Using this, then we also obtain that if $\alpha$ is an infinite cardinal and $2\leqslant\beta\leqslant\alpha$, then $\beta^{\alpha}=2^{\alpha}$.
\end{remark}

The following result shows that the set of wild primes of a direct product ring is either empty or uncountable, according to whether the index set is finite or infinite.

\begin{corollary}\label{Corollary 2 card} For $R=\prod\limits_{i\in S}R_{i}$ the following statements are equivalent. \\
$\mathbf{(i)}$ $R$ has a wild prime. \\
$\mathbf{(ii)}$ The index set $S$ is infinite. \\
$\mathbf{(iii)}$ The set of wild primes of $R$ has the cardinality $\geqslant2^{2^{|S|}}$.
\end{corollary}

\begin{proof} The implications (i)$\Rightarrow$(ii) and (iii)$\Rightarrow$(i) are clear. \\
(ii)$\Rightarrow$(iii): It is well known that $\Spec(\mathcal{P}(S))$ is the Stone-\v{C}ech compactification of the discrete space $S$. Therefore by \cite[Theorem 3.58]{Hindman-Strauss} or by \cite[Theorem on p.71]{Walker}, $|\Spec(\mathcal{P}(S))|=2^{2^{|S|}}$. It is clear that the set of tame primes of $\mathcal{P}(S)$ has the cardinality $|S|$. Then by Cantor's theorem (which asserts that $\alpha<2^{\alpha}$ for any cardinal $\alpha$), we have $|S|<2^{|S|}<2^{2^{|S|}}$. Next by using Remark \ref{Remark ii}, we obtain that the set of wild primes of $\mathcal{P}(S)$ has the cardinality $2^{2^{|S|}}$. Then by using Theorem \ref{Theorem T-D}, the desired conclusion is deduced.
\end{proof}

\begin{example} For $R=\prod\limits_{k\in S}R_{k}$ the map $\Spec(R)\rightarrow\Spec(\mathcal{P}(S))$ given by $\mathfrak{p}\mapsto M_{\mathfrak{p}}=\{A\in\mathcal{P}(S): \omega_{A}\in\mathfrak{p}\}$ is continuous and preserves wild primes where $\omega_{A}:=(r_{k})\in R$ with $r_{k}=1$ for $k\in A$ and otherwise $r_{k}=0$. But this map is not necessarily injective, because (for finite $S$ the assertion is clear) if $S$ is infinite and the map is injective then by Corollary \ref{Corollary 2 card}, the cardinality of $\Spec(R)$ will be $2^{2^{|S|}}$ which is impossible, because the cardinality of $\Spec(R)$ is strictly greater than the cardinality of each $\Spec(R_{k})$ and by \cite[Lemma 4.1]{Tarizadeh-Sharma} we may find a (Boolean) ring $R_{k}$ for some $k$ such that the cardinality of $\Spec(R_{k})>2^{2^{|S|}}$. Hence, this map is not necessarily injective. It is also a left inverse of the map in Theorem \ref{Theorem T-D}.
\end{example}

Note that if $\phi:A\rightarrow B$ is a surjective ring map and
$\phi^{-1}(E)$ is a (prime) ideal of $A$ for some subset $E$ of $B$, then $E$ is a (prime) ideal of $B$.

By a minimal tame prime of $R=\prod\limits_{k\in S}R_{k}$ we mean a tame prime of $R$ which is also a minimal prime of $R$. It can be easily seen that the minimal tame primes of $R$ are precisely of the form $\pi^{-1}_{k}(\mathfrak{p})$ where $\mathfrak{p}$ is a minimal prime of $R_{k}$ for some $k$. By a minimal wild prime of $R$ we mean a wild prime of $R$ which is also a minimal prime of $R$. The maximal tame prime and maximal wild prime notions are defined similarly. It can be seen that $P$ is a minimal tame prime of $R$ if and only if $P$ is a minimal element in the set of tame primes of $R$. Similar equivalences hold for minimal wild prime, maximal tame prime and maximal wild prime. It can be also seen that the maximal tame primes of $R$ are precisely of the form $\pi^{-1}_{k}(\mathfrak{m})$ where $\mathfrak{m}$ is a maximal ideal of $R_{k}$ for some $k$.

\begin{theorem}\label{Theorem 9 nine-dokuz} If the index set $S$ is infinite and each $R_{k}$ is an integral domain, then the set of minimal wild primes of $R=\prod\limits_{k\in S}R_{k}$ has the cardinality $2^{2^{|S|}}$.
\end{theorem}

\begin{proof} It can be shown that $\Min(R)$, the space of minimal prime ideals of $R$, is the Stone-\v{C}ech compactification of the discrete space $S$. Therefore by \cite[Theorem 3.58]{Hindman-Strauss} or by \cite[Theorem on p.71]{Walker}, $\Min(R)$ has the cardinality $2^{2^{|S|}}$. It is clear that the set of minimal tame primes of $R$ has the cardinality $|S|$. Then by using Cantor's Theorem and Remark \ref{Remark ii}, we get that the set of minimal wild primes of $R=\prod\limits_{k\in S}R_{k}$ has the cardinality $2^{2^{|S|}}$.
\end{proof}

\begin{theorem}\label{Theorem 10 ten} If the index set $S$ is infinite and each $R_{k}$ is a local ring, then the set of maximal wild primes of $R=\prod\limits_{k\in S}R_{k}$ has the cardinality $2^{2^{|S|}}$.
\end{theorem}

\begin{proof} It can be shown that $\Max(R)$, the space of maximal ideals of $R$, is the Stone-\v{C}ech compactification of the discrete space $S$. Then the remainder of the argument is exactly as the above proof, with taking into account that the set of maximal tame primes of $R$ has the cardinality $|S|$.
\end{proof}

Our next goal is to study the nilpotent elements of an infinite direct product ring and its quotient ring modulo the direct sum ideal. First note that the Jacobson radical is well behaved with the direct products. More precisely, for the ring $R=\prod\limits_{k\in S}R_{k}$ we have
$\mathfrak{J}(R)=\prod\limits_{k\in S}\mathfrak{J}(R_{k})$. But the nil-radical, unlike the Jacobson radical, is not well behaved with the direct products. In fact, $\mathfrak{N}(R)\subseteq\prod\limits_{k\in S}\mathfrak{N}(R_{k})$. Also $\bigoplus\limits_{k\in S}\mathfrak{N}(R_{k})\subseteq\mathfrak{N}(R)$. These inclusions can be strict. For example, in the ring $\prod\limits_{n\geqslant1}\mathbb{Z}/p^{n}\mathbb{Z}$ with $p$ a prime number, the sequence $(p+p^{n}\mathbb{Z})_{n\geqslant1}$ is a member of $\prod
\limits_{n\geqslant1}\mathfrak{N}(\mathbb{Z}/p^{n}\mathbb{Z})$ but this element is not nilpotent. To see the strictness of the second inclusion, consider the element $b=(b_{n})$ with $b_{1}=0$ and $b_{n}=p^{n-1}+p^{n}\mathbb{Z}$ for all $n\geqslant2$, then $b^{2}=0$ but $b\notin\bigoplus\limits_{n\geqslant1}
\mathfrak{N}(\mathbb{Z}/p^{n}\mathbb{Z})$.

For each $k\in S$, let $\mathfrak{p}_{k}$ be a fixed prime ideal of a ring $R_{k}$ and let $T$ be the set of all $r=(r_{k})$ in $R=\prod\limits_{k\in S}R_{k}$ such that the set $\Omega(r)=\{k\in S: r_{k}\notin\mathfrak{p}_{k}\}$ is cofinite (i.e., its complement $\Omega(r)^{c}=S\setminus\Omega(r)$ is finite). It is clear that $\Omega(ab)=\Omega(a)\cap\Omega(b)$ for all $a,b\in R$. In particular, $T$ is a multiplicative subset of $R$. For each $k\in S$, $1-e_{k}\in T$. So if $\mathfrak{p}$ is a prime ideal of $R$ with $\mathfrak{p}\cap T=\emptyset$, then $\mathfrak{p}$ is a wild prime of $R$.

In the following results, by $U$ we mean the set of all $r=(r_{k})$ in
$R=\prod\limits_{k\in S}R_{k}$ such that its support $S(r)=\{k\in S: r_{k}\neq0\}$ is cofinite.

\begin{lemma}\label{Theorem 6 imbed} If each $R_{k}$ is an integral domain, then the ring $R=\prod\limits_{k\in S}R_{k}$ modulo $\bigoplus\limits_{k\in S}R_{k}$ can be canonically imbedded in $U^{-1}R$.
\end{lemma}

\begin{proof} We observed that $U$ is a multiplicative subset of $R$. We show that $\Ker\pi=\bigoplus\limits_{k\in S}R_{k}$ where $\pi:R\rightarrow U^{-1}R$ is the canonical ring map. If  $a\in\Ker\pi$ then $ab=0$ for some $b\in U$. This yields that $S(a)\subseteq S(b)^{c}$ which is finite. Thus $a\in I:=\bigoplus\limits_{k\in S}R_{k}$. To see the reverse inclusion, take
$a\in I$. Then consider the element $b=(b_{k})\in R$ such that $b_{k}$ is either $0$ or $1$, according to whether $k\in S(a)$ or $k\notin S(a)$. Then clearly $ab=0$, and $b\in U$ because $S(b)^{c}=S(a)$ is finite. So we obtain an injective morphism of rings $\phi:R/I\rightarrow U^{-1}R$ given by $a+I\mapsto a/1$.
\end{proof}

\begin{remark}\label{Remark iv lying over} Remember that if $\phi:A\rightarrow B$ is an injective ring map and $\mathfrak{p}$ is a minimal prime ideal of $A$, then there exists a prime ideal $\mathfrak{q}$ of $B$ which laying over $\mathfrak{p}$, i.e., $\mathfrak{p}=\phi^{-1}(\mathfrak{q})$. Indeed, consider the following commutative (pushout) diagram: $$\xymatrix{
A\ar[r]^{\phi}\ar[d]^{\pi} &B\ar[d]^{}\\A_{\mathfrak{p}}\ar[r]^{}&
A_{\mathfrak{p}}\otimes_{A}B}$$ where $\pi$ and the unnamed arrows are the canonical maps. Clearly $A_{\mathfrak{p}}\otimes_{A}B\simeq B_{\mathfrak{p}}$ is a nonzero ring, since $\phi$ is injective. So it has a prime ideal $P$. The contraction of $P$ under the canonical ring map $A_{\mathfrak{p}}\rightarrow A_{\mathfrak{p}}\otimes_{A}B$ equals $\mathfrak{p}A_{\mathfrak{p}}$, because $\mathfrak{p}$ is a minimal prime of $A$. Let $\mathfrak{q}$ be the contraction of $P$ under the canonical ring map $B\rightarrow A_{\mathfrak{p}}\otimes_{A}B$ which is a prime ideal of $B$. Now we have $\phi^{-1}(\mathfrak{q})=
\pi^{-1}(\mathfrak{p}A_{\mathfrak{p}})=\mathfrak{p}$. Then using the axiom of choice, we obtain an injective map $\Min(A)\rightarrow\Spec(B)$.
\end{remark}

\begin{corollary} If each $R_{k}$ is an integral domain and $\mathfrak{p}$ is a minimal wild prime of $R=\prod\limits_{k\in S}R_{k}$, then $\mathfrak{p}\cap U=\emptyset$.
\end{corollary}

\begin{proof} By Lemma \ref{Theorem 6 imbed}, we have an injective ring map $\phi:R/I\rightarrow U^{-1}R$ where $I=\bigoplus\limits_{k\in S}R_{k}$. Clearly $\mathfrak{p}/I$ is a minimal prime ideal of $R/I$. So by Remark \ref{Remark iv lying over}, there exists a prime ideal $\mathfrak{q}$ of $R$ such that $\mathfrak{q}\cap U=\emptyset$ and $\mathfrak{p}/I=\phi^{-1}(U^{-1}\mathfrak{q})$. But the composition of $\phi$ with the canonical map $R\rightarrow R/I$ gives us the canonical map $R\rightarrow U^{-1}R$. It follows that $\mathfrak{p}=\mathfrak{q}$. \\
As a second proof, suppose $a\in\mathfrak{p}\cap U$. It is well known that $ab$ is nilpotent for some $b\in R\setminus\mathfrak{p}$. It follows that $ab=0$, because $R$ is a reduced ring. Thus $S(b)\subseteq S\setminus S(a)$ which is finite. This shows that $b\in I\subseteq\mathfrak{p}$ which is a contradiction.
\end{proof}

In the following two results, recall that $T$ denotes the set of all $r=(r_{k})$ in $R=\prod\limits_{k\in S}R_{k}$ such that
$\Omega(r)=\{k\in S: r_{k}\notin\mathfrak{m}_{k}\}$ is cofinite.

\begin{theorem}\label{Theorem 8 eight} If each $R_{k}$ is a local ring with the maximal ideal $\mathfrak{m}_{k}$, then the ring $R=\prod\limits_{k\in S}R_{k}$ modulo $\bigoplus\limits_{k\in S}R_{k}$ is canonically isomorphic to $T^{-1}R$.
\end{theorem}

\begin{proof} We first show that $\Ker\pi=\bigoplus\limits_{k\in S}R_{k}$ where $\pi:R\rightarrow T^{-1}R$ is the canonical ring map. If $a\in\Ker\pi$ then $ab=0$ for some $b\in T$. This yields that $S(a)=\{k\in S: a_{k}\neq0\}\subseteq
\Omega(b)^{c}$ which is finite. So $a\in I:=\bigoplus\limits_{k\in S}R_{k}$. To see the reverse inclusion, take
$a\in I$. Then consider the element $b=(b_{k})\in R$ such that $b_{k}$ is either $0$ or $1$, according to whether $k\in S(a)$ or $k\notin S(a)$. Then clearly $ab=0$, and $b\in T$ because $\Omega(b)^{c}=S(a)$ is finite. Thus we obtain an injective morphism of rings $\phi:R/I\rightarrow T^{-1}R$ given by $a+I\mapsto a/1$. The image of each $b=(b_{k})\in T$ under the canonical ring map $R\rightarrow R/I$ is invertible, because consider the element $b'=(b'_{k})\in R$ where $b'_{k}:=b^{-1}_{k}$ if $k\in\Omega(b)$ and otherwise $b'_{k}:=0$, then $1-bb'=\sum\limits_{k\in\Omega(b)^{c}}e_{k}\in I=\Ker\pi$. Thus $c(1-bb')=0$ for some $c\in T$. It follows that $\phi(ab')=a/b$.
Hence, $\phi$ is surjective.
\end{proof}

\begin{corollary} If each $(R_{k},\mathfrak{m}_{k})$ is a local ring, then a prime ideal $\mathfrak{p}$ of $R=\prod\limits_{k\in S}R_{k}$ is a wild prime if and only if $\mathfrak{p}\cap T=\emptyset$.
\end{corollary}

\begin{proof}  The implication ``$\Rightarrow$'' is an immediate consequence of the above result. The reverse implication holds more generally.
\end{proof}

The following result also immediately follows from the above theorem.

\begin{corollary} If each $R_{k}$ is a field, then the ring $R=\prod\limits_{k\in S}R_{k}$ modulo $\bigoplus\limits_{k\in S}R_{k}$ is canonically isomorphic to $U^{-1}R$.
\end{corollary}

The following result is well known (see \cite[Theorem 3.4]{Gilmer-Heinzer} and \cite[Proposition 2.6]{Maroscia}). We only give a new proof for the equivalence (i)$\Leftrightarrow$(ii).

\begin{theorem}\label{Theorem 2 Krull} If each $R_i$ is a zero-dimensional ring, then for $R=\prod\limits_{i\in S}R_{i}$ the following assertions are equivalent. \\
$\mathbf{(i)}$ $\dim(R)=0$. \\
$\mathbf{(ii)}$ $\mathfrak{N}(R)=\prod\limits_{i\in S}\mathfrak{N}(R_{i})$. \\
$\mathbf{(iii)}$ $\dim(R)$ is finite.
\end{theorem}

\begin{proof} (i)$\Rightarrow$(ii): By hypothesis, $\mathfrak{N}(R)=\mathfrak{J}(R)=\prod\limits_{i\in S}\mathfrak{J}(R_i)=\prod\limits_{i\in S}\mathfrak{N}(R_i)$. \\
(ii)$\Rightarrow$(i): For any ring $R$, we have $\dim(R)=\dim(R/\mathfrak{N}(R))$. By hypothesis, the ring $R/\mathfrak{N}(R)$ is canonically isomorphic to $R':=\prod\limits_{i\in S}R_i/\mathfrak{N}(R_i)$. But each $R_i/\mathfrak{N}(R_i)$ is a reduced zero-dimensional ring, and so it is a von-Neumann regular ring. It is easy to see that the every direct product of von-Neumann regular rings is von-Neumann regular, and every von-Neumann regular ring is zero-dimensional. Thus, $\dim(R)=\dim(R')=0$. \\
(i)$\Rightarrow$(iii): There is nothing to prove. \\
(iii)$\Rightarrow$(ii): See \cite[Theorem 3.4]{Gilmer-Heinzer}.
\end{proof}

\begin{corollary} If a ring $R$ has a principal maximal ideal $Rx$ such that $x$ is a non-zero-divisor of $R$, then the Krull dimensions of $\prod\limits_{n\geqslant1}R/(x^{n})$ and  $\prod\limits_{n\geqslant1}R$ are infinite.
\end{corollary}

\begin{proof} Each $a_{n}:=x+(x^{n})\in R/(x^{n})$ is nilpotent, whereas the element $a=(a_{1},a_{2},a_{3},\ldots)$ is not nilpotent, because if $a^{d}=0$ for some $d\geqslant1$, then $x^{d}\in(x^{d+1})$ and so $x^{d}=rx^{d+1}$ for some $r\in R$, but $x^{d}$ is non-zero-divisor so we get that $x$ is an invertible in $R$ which is a contradiction. Also each $R/(x^{n})$ is zero-dimensional. Thus by Theorem \ref{Theorem 2 Krull}, the Krull dimension of $\prod\limits_{n\geqslant1}R/(x^{n})$ is infinite. The canonical surjective ring map $\prod\limits_{n\geqslant1}R
\rightarrow\prod\limits_{n\geqslant1}R/(x^{n})$ induces an injective map between the corresponding prime spectra. It follows that the Krull dimension of $\prod\limits_{n\geqslant1}R$ is infinite.
\end{proof}

In particular, the Krull dimensions of $\prod\limits_{n\geqslant1}\mathbb{Z}/p^{n}\mathbb{Z}$ and $\prod\limits_{n\geqslant1}\mathbb{Z}$ are infinite.

As another application of Theorem \ref{Theorem 1 IDP}, we  give a more complete proof of our recent result \cite[Corollary 3.12]{Tarizadeh-Chen 2}. First note that by an avoidance ring we mean a ring $R$ such that every ideal $I$ of $R$ has the ideal avoidance property. This means that whenever $I_{1},\ldots,I_{n}$ are finitely many ideals of $R$ with $I \subseteq \bigcup\limits_{k=1}^{n}I_{k}$, then $I\subseteq I_{k}$ for some $k$.

\begin{theorem}\label{Theorem 3 avoidance} Let $(R_{i})_{i\in S}$ be a family of avoidance rings. If $R=\prod\limits_{i\in S}R_{i}$ modulo the ideal $I=\bigoplus\limits_{i\in S}R_{i}$ is an avoidance ring, then $R$ is an avoidance ring.
\end{theorem}

\begin{proof} It will be enough to show that $R$ satisfies \cite[Theorem 2.6(c)]{Quartararo-Butts} which asserts that a ring $R$ is an avoidance ring if and only if for each maximal ideal $M$ of $R$, either the field $R/M$ is infinite or $R_{M}$ is a B\'{e}zout ring (i.e. every finitely generated ideal is principal).
Let $M$ be a maximal ideal of $R$. If $I\subseteq M$ then by hypothesis and \cite[Theorem 2.6]{Quartararo-Butts}, either the field $(R/I)/(M/I)\simeq R/M$ is infinite or the localization $(R/I)_{M/I}\simeq(R/I)_{M}\simeq R_{M}/IR_{M}$ is a B\'{e}zout ring. But $I$ is generated by the idempotents $e_{k}$ with $k\in S$. So its extension $IR_{M}$ is generated by the elements $e_{k}/1$. Since each $1-e_{k}\in R\setminus M$, so $e_{k}/1=0$. Thus $IR_{M}=0$.
Hence, $R$ satisfies the condition (c) in this case.
Now assume $M$ does not contain $I$. Then by Theorem \ref{Theorem 1 IDP} (or, by Corollary \ref{Corollary I}), $M$ is a tame prime of $R$. So $M=\prod\limits_{i}M_{i}$ where $M_{k}=\pi_{k}(M)$ is a maximal ideal of $R_{k}$ for some $k$ and $M_{i}=R_{i}$ for all $i\neq k$. But if $R/M\simeq R_{k}/M_{k}$ is a finite field, then by \cite[Theorem 2.6]{Quartararo-Butts},  $R_{M}\simeq(R_{k})_{M_{k}}$ is a B\'{e}zout ring.
\end{proof}

\section{Tame and wild connected components}

In this section, we investigate the connected components of the prime spectrum of an infinite direct product of rings.

\begin{remark}\label{Remark 7 st} We will use the following set-theoretical observation in the next results. If $(A_{k})$ is a family of finite nonempty sets (that is, $1\leqslant|A_{k}|<\infty$ for all $k$) indexed by an infinite set $S$, then the disjoint union $\coprod\limits_{k\in S}A_{k}=\bigcup\limits_{k\in S}\{k\}\times A_{k}$ is in bijection with $S$. Indeed, for each $k\in S$ we may choose some $x_{k}\in A_{k}$, because it is nonempty. Thus
the map $S\rightarrow\coprod\limits_{k\in S}A_{k}$ given by $k\mapsto(k,x_{k})$ is injective. Hence, $|S|\leqslant\beta$ where $\beta$ denotes the cardinality of
$\coprod\limits_{k\in S}A_{k}$. To see the reverse inequality, we act as follows. Each $A_{k}$ is in bijection with a natural number $n_{k}\geqslant1$. Remember that every natural number $n$ is the set of all natural numbers strictly less than $n$, i.e., $0=\{\}$, $1=\{0\}$, $2=\{0,1\}$, $3=\{0,1,2\}$ and so on.
Thus $\coprod\limits_{k\in S}A_{k}$ is in bijection with $\coprod\limits_{k\in S}n_{k}=\bigcup\limits_{k\in S}\{k\}\times n_{k}\subseteq S\times\mathbb{N}$. Then using Remark \ref{Remark ii}, we have $\beta\leqslant|S|\cdot\aleph_{0}\leqslant |S|\cdot|S|=|S|$. Therefore, $\beta=|S|$.
\end{remark}

For any ring $R$ by $\mathcal{B}(R)=\{e\in R: e=e^{2}\}$ we mean the set of all idempotents of $R$ which is a commutative ring whose addition is $e\oplus e':=e+e'-2ee'$ and whose multiplication is $e\cdot e'=ee'$. We call $\mathcal{B}(R)$ the Boolean ring of $R$. For more information on this ring we refer the interested reader to \cite{Tarizadeh-Taheri}.

\begin{theorem}\label{Thm ccc} If  $R=\prod\limits_{k\in S}R_{k}$ with $S$ infinite, then the set of connected components of $\Spec(R)$ has the cardinality $\geqslant2^{2^{|S|}}$. If moreover, each $R_{k}$ has only finitely many idempotents then the equality holds.
\end{theorem}

\begin{proof} By \cite[Theorem 4.1]{Tarizadeh-Taheri}, the set of connected components of $\Spec(R)$ endowed with the quotient topology is canonically homeomorphic to the prime spectrum of $\mathcal{B}(R)$. Also each $R_{k}$ is a nonzero ring, so the ring map $\mathcal{P}(S)\simeq\prod\limits_{k\in S}\mathbb{Z}_{2}\rightarrow\mathcal{B}(R)=\prod\limits_{k\in S}\mathcal{B}(R_{k})$ induced by the canonical injective ring maps $\mathbb{Z}_{2}\rightarrow\mathcal{B}(R_{k})$ is injective. Then using Remark \ref{Remark iv lying over} and the facts that the ring $\mathcal{P}(S)$ is zero dimensional and $\Spec(\mathcal{P}(S))$ is the Stone-\v{C}ech compactification of the discrete space $S$, we conclude that the prime spectrum of $\mathcal{B}(R)$ and so the space of connected components of $\Spec(R)$ have the cardinality $\geqslant2^{2^{|S|}}$. If each $R_{k}$ has finitely many idempotents, then using Chinese remainder theorem we observe that each finite Boolean ring $\mathcal{B}(R_{k})$ is canonically isomorphic to $\prod\limits_{i=1}^{n_{k}}\mathbb{Z}_{2}$ where the natural number $n_{k}\geqslant1$ denotes the number of prime ideals of $\mathcal{B}(R_{k})$. It follows that $\mathcal{B}(R)\simeq\prod\limits_{k\in S}(\prod\limits_{i=1}^{n_{k}}\mathbb{Z}_{2})\simeq\mathcal{P}(X)$ where $X:=\coprod\limits_{k\in S}\Spec(\mathcal{B}(R_{k}))$.
By Remark \ref{Remark 7 st}, the set $X$ is in bijection with $S$. Hence, $\Spec(\mathcal{B}(R))$ is the Stone-\v{C}ech compactification of the discrete space $S$. Thus the prime spectrum of $\mathcal{B}(R)$ and so the space of connected components of $\Spec(R)$ have the cardinality $2^{2^{|S|}}$.
\end{proof}

\begin{lemma}\label{Lemma 3 nice ob} Let $R=\prod\limits_{k\in S}R_{k}$. Then every connected component of $\Spec(R)$ is contained in the set of tame primes or in the set of wild primes of $R$.
\end{lemma}

\begin{proof} Let $C$ be a connected component of $\Spec(R)$. Suppose $\mathfrak{p}$ is a tame prime and $\mathfrak{q}$ is a wild prime of $R$ so that $\mathfrak{p},\mathfrak{q}\in C$. Then there exists some $k\in S$ such that $\mathfrak{p}\in U:=C\cap D(e_{k})$ and $\mathfrak{q}\in V:=C\cap V(e_{k})=C\cap D(1-e_{k})$. Clearly $C$ is covered by the nonempty disjoint open subsets $U$ and $V$ which violates the connectedness of $C$.
\end{proof}

\begin{remark} It is well known that for any ring $R$, then every connected component of the prime spectrum $\Spec(R)$ is precisely of the form $V(M)$ where $M$ is a max-regular ideal of $R$. Remember that every maximal element of the set of proper regular ideals of $R$ is called a max-regular ideal of $R$. Here, by a regular ideal we mean an ideal of $R$ which is generated by a set of idempotents of $R$. In fact, the map $M\mapsto V(M)$ is a bijection from the set of max-regular ideals of $R$ onto the set of connected components of its prime spectrum. For the details see e.g. \cite[Theorem 3.17]{Tarizadeh 3}. If $\mathfrak{p}$ is a prime ideal of $R$, then clearly $(e\in\mathfrak{p}: e=e^{2})$ is a max-regular ideal of $R$. If $I$ and $J$ are regular ideals of $R$ such that $V(I)=V(J)$, then $I=J$.
\end{remark}

Now Lemma \ref{Lemma 3 nice ob} together with the above remark leads us to the following definition.

\begin{definition} By a \emph{tame max-regular ideal of $R=\prod\limits_{k\in S}R_{k}$} we mean a max-regular ideal $M$ of $R$ such that $V(M)$ is contained in the set of tame primes of $R$. Dually, by a \emph{wild max-regular ideal of $R$} we mean a max-regular ideal $M$ of $R$ such that $V(M)$ is contained in the set of wild primes of $R$.
\end{definition}

By \cite[Theorem 4.1]{Tarizadeh-Taheri}, tame (resp. wild) max-regular ideals of $R=\prod\limits_{k\in S}R_{k}$ are in one-to-one correspondence with tame (resp. wild) primes of the Boolean ring $\mathcal{B}(R)=\prod\limits_{k\in S}\mathcal{B}(R_{k})$. Also, we have precisely two types of connected components in $\Spec(R)$: tame connected components and wild connected components.
In the following result, the structure of tame connected components is characterized.

\begin{theorem}\label{Theorem tmri ng} Tame max-regular ideals of $R=\prod\limits_{i\in S}R_{i}$ are precisely of the form $\pi_{k}^{-1}(M)$ where $M$ is a max-regular ideal of $R_{k}$ for some $k$.
\end{theorem}

\begin{proof} We first show that $\pi_{k}^{-1}(M)$ is a tame max-regular ideal of $R$. If $a\in R_{k}$ then $ae_{k}$ denotes the element in $R$ which has $a$ in the $k$ component and zero in all other components. Then we claim that $\pi_{k}^{-1}(M)=R(1-e_{k})+(ae_{k}: a\in M, a=a^{2})$. The inclusion $R(1-e_{k})+(ae_{k}: a\in M, a=a^{2})\subseteq M$ is obvious. To see the reverse inclusion, take $r=(r_{i})\in\pi_{k}^{-1}(M)$. Then consider the elements $r'=(r'_{i})$ and $r''=(r''_{i})$ in $R$ with $r'_{i}:=r_{i}$ and $r''_{i}:=0$ for all $i\neq k$, $r'_{k}=0$ and $r''_{k}=r_{k}$. Since $r_{k}\in M$, so $r''\in(ae_{k}: a\in M, a=a^{2})$. We also have $r=r'+r''$ where $r'=r'(1-e_{k})\in R(1-e_{k})$. This establishes the claim. This shows that $M':=\pi^{-1}(M)$ is a proper regular ideal of $R$ and $V(M')$ is contained in the set of tame primes of $R$. The map $R/M'\rightarrow R_{k}/M$ induced by $\pi_{k}:R\rightarrow R_{k}$ is an isomorphism of rings. Then using \cite[Lemma 3.19]{Tarizadeh 3} which asserts that a proper regular ideal $I$ of a ring $R$ is max-regular if and only if $R/I$ has no nontrivial idempotents, we have  $R/M'$ has no nontrivial idempotents and so $M'$ is a max-regular ideal of $R$. Conversely, let $\mathcal{M}$ be a tame max-regular ideal of $R$. Since $\mathcal{M}\neq R$, then $\mathcal{M}\subseteq P$ for some prime ideal $P$ of $R$. Thus $P$ is a tame prime of $R$. Hence, $P=\pi_{k}^{-1}(\mathfrak{p})$ where $\mathfrak{p}$ is a prime ideal of $R_{k}$ for some $k$. Note that $N:=(e\in\mathfrak{p}: e=e^{2})$ is a max-regular ideal of $R_{k}$. Clearly $\mathcal{M}\subseteq\pi_{k}^{-1}(N)$. But in the above, we observed that $\pi_{k}^{-1}(N)$ is a proper regular ideal of $R$. Hence, $\mathcal{M}=\pi_{k}^{-1}(N)$.
\end{proof}

\begin{remark}\label{Remark 8 chg} Regarding Theorem \ref{Theorem tmri ng}, note that if $M\subset R_{k}$ and $N\subset R_{d}$ are max-regular ideals with $\pi_{k}^{-1}(M)=\pi_{d}^{-1}(N)$, then $k=d$ and $M=N$. For any ring $R$, the set of all max-regular ideals of $R$ is denoted by $\Sp(R)$.
Thus by Theorem \ref{Theorem tmri ng}, the map from the disjoint union $\coprod\limits_{k\in S}\Sp(R_{k})$ onto the set of tame max-regular ideals of $R=\prod\limits_{k\in S}R_{k}$ given by  $(k,M)\mapsto\pi_{k}^{-1}(M)$ is bijective.
\end{remark}

\begin{corollary} Let $R=\prod\limits_{i\in S}R_{i}$. Then $V(1-e_{k})$ is a connected component of $\Spec(R)$ if and only if $R_{k}$ has no nontrivial idempotents.
\end{corollary}

\begin{proof} If $R_{k}$ has no nontrivial idempotents, then the zero ideal of $R_{k}$ is a max-regular ideal. Then by Theorem \ref{Theorem tmri ng}, $\pi_{k}^{-1}(0)=R(1-e_{k})$ is a max-regular ideal of $R$. Hence, $V(1-e_{k})$ is a connected component of $\Spec(R)$. To see the converse, first note that for any ring $R$, if $I$ is a regular ideal of $R$ such that $V(I)$ is a connected component of $\Spec(R)$, then $I$ is a max-regular ideal of $R$.
Hence, $R(1-e_{k})$ is a max-regular ideal of $R$ and so $R/R(1-e_{k})\simeq R_{k}$ has no nontrivial idempotents.
\end{proof}

\begin{corollary}\label{Corollary 4 dort-four} Let $R=\prod\limits_{k\in S}R_{k}$. If $S$ is infinite and each $R_{k}$ has finitely many idempotents, then the set of wild max-regular ideals of $R$ has the cardinality $2^{2^{|S|}}$.
\end{corollary}

\begin{proof} By hypothesis, each $\Spec(R_{k})$ has finitely many connected components. Hence, each $\Sp(R_{k})$ is a nonempty finite set.
Thus by Remarks \ref{Remark 7 st} and $\ref{Remark 8 chg}$, the set of tame max-regular ideals of $R$ has the cardinality $|S|$. By Theorem \ref{Thm ccc}, the set of max-regular ideals of $R$ has the cardinality $2^{2^{|S|}}$. Then by Cantor's theorem: $|S|<2^{|S|}<2^{2^{|S|}}$ and using Remark \ref{Remark ii}, we obtain that the set of wild max-regular ideals of $R$ has the cardinality $2^{2^{|S|}}$.
\end{proof}

\begin{remark} If $S$ is infinite, then the direct sum ideal $\bigoplus\limits_{k\in S}R_{k}$ is contained in every wild max-regular ideal of $R=\prod\limits_{i\in S}R_{i}$. Also, if $e\in R$ is an idempotent such that both $\Supp(e)$ and $\Supp(1-e)$ are infinite, then $\bigoplus\limits_{k\in S}R_{k}+Re$ is a proper regular ideal of $R$, and so it is contained in a wild max-regular ideal of $R$.
But, exactly like wild primes, all of the wild max-regular ideals are non-constructive and hence we cannot give an explicit example of them. For instance, in a power set ring $\mathcal{P}(S)$, wild max-regular ideals and wild primes are the same. In fact, in a Boolean ring or more generally in a von-Neumann regular ring, max-regular ideals and prime ideals are the same, because in such a ring every ideal is a regular ideal.
\end{remark}

\textbf{Acknowledgments.} We would like to give sincere thanks to Professor Pierre Deligne who generously shared with us his very valuable and excellent ideas.



\end{document}